\documentclass[11pt]{article}
\usepackage[margin=1in]{geometry} 
\usepackage{amssymb,amsfonts,amsmath,bbm,mathrsfs,stmaryrd,mathtools}
\usepackage{xcolor}
\usepackage{url}

\usepackage{extarrows}

\usepackage[shortlabels]{enumitem}
\usepackage{tensor}

\usepackage{xr}
\usepackage[T1]{fontenc}

\usepackage[colorlinks,
linkcolor=black!75!red,
citecolor=blue,
pdftitle={},
pdfproducer={pdfLaTeX},
pdfpagemode=None,
bookmarksopen=true,
bookmarksnumbered=true,
backref=page]{hyperref}

\usepackage{tikz}
\usetikzlibrary{arrows,calc,decorations.pathreplacing,decorations.markings,intersections,shapes.geometric,through,fit,shapes.symbols,positioning,decorations.pathmorphing}

\usepackage{braket}

\usepackage[amsmath,thmmarks,hyperref]{ntheorem}
\usepackage{cleveref}

\newcommand\nthalias[1]{\AddToHook{env/#1/begin}{\crefalias{lemma}{#1}}}

\nthalias{definition}
\nthalias{example}
\nthalias{examples}
\nthalias{remark}
\nthalias{remarks}
\nthalias{convention}
\nthalias{notation}
\nthalias{construction}
\nthalias{theoremN}
\nthalias{propositionN}
\nthalias{corollaryN}
\nthalias{lemma}
\nthalias{proposition}
\nthalias{corollary}
\nthalias{theorem}
\nthalias{conjecture}
\nthalias{question}
\nthalias{assumption}

\creflabelformat{enumi}{#2#1#3}

\crefname{section}{Section}{Sections}
\crefformat{section}{#2Section~#1#3} 
\Crefformat{section}{#2Section~#1#3} 

\crefname{subsection}{\S}{\S\S}
\AtBeginDocument{%
  \crefformat{subsection}{#2\S#1#3}%
  \Crefformat{subsection}{#2\S#1#3}%
}

\crefname{subsubsection}{\S}{\S\S}
\AtBeginDocument{%
  \crefformat{subsubsection}{#2\S#1#3}%
  \Crefformat{subsubsection}{#2\S#1#3}%
}

\theoremstyle{plain}

\newtheorem{lemma}{Lemma}[section]

\newtheorem{theorem}[lemma]{Theorem}

\theoremstyle{plain}
\theoremnumbering{Alph}

\theoremstyle{plain}
\theorembodyfont{\upshape}
\theoremsymbol{\ensuremath{\blacklozenge}}

\newtheorem{remark}[lemma]{Remark}

\crefname{definition}{definition}{definitions}
\crefformat{definition}{#2definition~#1#3} 
\Crefformat{definition}{#2Definition~#1#3} 

\crefname{ex}{example}{examples}
\crefformat{example}{#2example~#1#3} 
\Crefformat{example}{#2Example~#1#3} 

\crefname{exs}{example}{examples}
\crefformat{examples}{#2example~#1#3} 
\Crefformat{examples}{#2Example~#1#3} 

\crefname{remark}{remark}{remarks}
\crefformat{remark}{#2remark~#1#3} 
\Crefformat{remark}{#2Remark~#1#3} 

\crefname{remarks}{remark}{remarks}
\crefformat{remarks}{#2remark~#1#3} 
\Crefformat{remarks}{#2Remark~#1#3} 

\crefname{convention}{convention}{conventions}
\crefformat{convention}{#2convention~#1#3} 
\Crefformat{convention}{#2Convention~#1#3} 

\crefname{notation}{notation}{notations}
\crefformat{notation}{#2notation~#1#3} 
\Crefformat{notation}{#2Notation~#1#3} 

\crefname{table}{table}{tables}
\crefformat{table}{#2table~#1#3} 
\Crefformat{table}{#2Table~#1#3}

\crefname{lemma}{lemma}{lemmas}
\crefformat{lemma}{#2lemma~#1#3} 
\Crefformat{lemma}{#2Lemma~#1#3} 

\crefname{proposition}{proposition}{propositions}
\crefformat{proposition}{#2proposition~#1#3} 
\Crefformat{proposition}{#2Proposition~#1#3} 

\crefname{propositionN}{proposition}{propositions}
\crefformat{propositionN}{#2proposition~#1#3} 
\Crefformat{propositionN}{#2Proposition~#1#3} 

\crefname{corollary}{corollary}{corollaries}
\crefformat{corollary}{#2corollary~#1#3} 
\Crefformat{corollary}{#2Corollary~#1#3} 

\crefname{corollaryN}{corollary}{corollaries}
\crefformat{corollaryN}{#2corollary~#1#3} 
\Crefformat{corollaryN}{#2Corollary~#1#3} 

\crefname{theorem}{theorem}{theorems}
\crefformat{theorem}{#2theorem~#1#3} 
\Crefformat{theorem}{#2Theorem~#1#3} 

\crefname{theoremN}{theorem}{theorems}
\crefformat{theoremN}{#2theorem~#1#3} 
\Crefformat{theoremN}{#2Theorem~#1#3} 

\crefname{enumi}{}{}
\crefformat{enumi}{#2#1#3}
\Crefformat{enumi}{#2#1#3}

\crefname{assumption}{assumption}{Assumptions}
\crefformat{assumption}{#2assumption~#1#3} 
\Crefformat{assumption}{#2Assumption~#1#3} 

\crefname{construction}{construction}{Constructions}
\crefformat{construction}{#2construction~#1#3} 
\Crefformat{construction}{#2Construction~#1#3} 

\crefname{question}{question}{Questions}
\crefformat{question}{#2question~#1#3} 
\Crefformat{question}{#2Question~#1#3} 

\crefname{equation}{}{}
\crefformat{equation}{(#2#1#3)} 
\Crefformat{equation}{(#2#1#3)}

\numberwithin{equation}{section}

\theoremstyle{nonumberplain}
\theoremsymbol{\ensuremath{\blacksquare}}

\newtheorem{proof}{Proof}
\newcommand\pf[1]{\newtheorem{#1}{Proof of \Cref{#1}}}

\newcommand\cC{{\mathcal C}}
\newcommand\cD{{\mathcal D}}

\newcommand\cM{{\mathcal M}}

\newcommand\cO{{\mathcal O}}

\newcommand\fM{{\mathfrak M}}

\newcommand\fa{{\mathfrak a}}

\newcommand\fm{{\mathfrak m}}

\newcommand\fp{{\mathfrak p}}

\DeclareMathOperator{\id}{id}

\DeclareMathOperator{\rep}{\mathrm{Rep}}
\DeclareMathOperator{\supp}{\mathrm{supp}}
\DeclareMathOperator{\spec}{\mathrm{Spec}}

\newcommand{\define}[1]{{\em #1}}

\newcommand\spr[1]{\cite[\href{https://stacks.math.columbia.edu/tag/#1}{Tag {#1}}]{stacks-project}}

\renewcommand{\square}{\mathrel{\Box}}

\newcommand{\xrightarrowdbl}[2][]{%
  \xrightarrow[#1]{#2}\mathrel{\mkern-14mu}\rightarrow
}

\title{Finite-dimensional modules over associative equivariant map algebras}
\author{Alexandru Chirvasitu}

\begin{document}

\date{}

\newcommand{\Addresses}{{
  \bigskip
  \footnotesize

  \textsc{Department of Mathematics, University at Buffalo}
  \par\nopagebreak
  \textsc{Buffalo, NY 14260-2900, USA}  
  \par\nopagebreak
  \textit{E-mail address}: \texttt{achirvas@buffalo.edu}


}}

\maketitle

\begin{abstract}
  Let $X$ and $\mathfrak{a}$ be an affine scheme and (respectively) a finite-dimensional associative algebra over an algebraically-closed field $\Bbbk$, both equipped with actions by a linearly-reductive linear algebraic group $G$. We describe the simple finite-dimensional modules over the algebra of $G$-equivariant maps $X\to \mathfrak{a}$ in terms of the representation theory of the fixed-point subalgebras $\mathfrak{a}^x:=\mathfrak{a}^{G_x}\le \mathfrak{a}$, $G_x$ being the respective isotropy groups of closed-orbit $k$-points $x\in X$. This answers a question of E. Neher and A. Savage, extending an analogous result for (also linearly-reductive) finite-group actions. Moreover, the full category of finite-dimensional modules admits a direct-sum decomposition indexed by closed orbits. 
\end{abstract}


\vspace{.5cm}

\noindent{MSC 2020: 14L17; 14L30; 16D60; 16T05; 18M05; 16T15; 14A15; 18C40
  
}


\section*{Introduction}

The note is motivated by a number of questions raised in \cite[\S 4.1]{NehSav13} in the process of studying the \emph{equivariant map algebras} that form the object of \cite{NehSavSen12} and feature in various guises much other literature: \cite{MR3317800,MR3484367,MR3205779} and their references, for instance. The setup, briefly, is as follows.

\begin{itemize}[wide]
\item Working throughout over an algebraically closed field $\Bbbk$ of arbitrary characteristic, consider a fixed commutative algebra $A$ with associated affine scheme $X:=\spec(A)$.

\item $\fa$ is a finite-dimensional algebra in the general sense of that term (at this early stage in the discussion): vector space equipped with a number of tensors satisfying a number of equational constraints ($\fa$ will be unital associative in the present work, and is mostly a Lie algebra in much of the work cited above).

\item $A$ and $\fa$ are both acted upon by a \emph{linear algebraic group} \cite[Remark 4.11]{mln_alg-gp_2017} $G$ (assumed finite in the cited sources but not here), mostly assumed \emph{linearly reductive} (\cite[Definition 12.52]{mln_alg-gp_2017}, \cite[\S 1.1, Definition 1.4]{fkm}) below.

\item The main object of study is the fixed-point subalgebra $\fM=(A\otimes\fa)^G\le A\otimes \fa$, i.e. the algebra of $G$-equivariant regular maps $X\to\fa$ (hence the name: equivariant map algebra). ``Object of study'' is understood representation-theoretically in much of the literature: classifying/describing appropriate classes of modules over $\fM$, whatever the phrase ``module'' might mean (depending on the structure $\fa$: Lie, associative, etc.). 
\end{itemize}

The material preceding it having focused on finite $G$ (regarded as a finite scheme with the set underlying $G$ as that of closed points) of order coprime to $\mathrm{char}\;\Bbbk$, \cite[Problem 4.1(a)]{NehSav13} proposes extending the discussion to broader classes of algebraic groups. This (for associative unital $fa$) is the focus of the present note. 

The module category over an algebra $B$ is denoted by ${_B}\cM$, while $\cM^C$ stands for the category of comodules over a coalgebra $C$. An extra `$f$' subscript, as in ${_B}\cM_f$, indicates finite-dimensional objects. All module structures are on the left and all comodule structures on the right, unless explicitly amended. For a point $x\in X(\Bbbk)$ we write $G_x\le G$ for its \emph{isotropy group} \cite[post Proposition 7.5]{mln_alg-gp_2017} and $\fa^x\le \fa$ for the subalgebra fixed by $G_x$. As explained e.g. in \cite[Definition 1.6]{NehSav13}, we have an evaluation map $\fM\to\fa^x$ obtained, as the name suggests, simply by evaluating a $G$-equivariant map $X\to\fa$ at $x\in X$. Simple finite-dimensional $\fM$-modules, then, are classifiable as perhaps expected (in a statement generalizing its finite-group counterpart \cite[Theorem 2.1]{NehSav13}). 

This notation in place, the classification of simple irreducible $\fM$-modules reads as follows. 

\begin{theorem}\label{th:main}
  Let $G$ be a smooth linearly-reductive linear algebraic group acting on an affine scheme $X=\spec(A)$ as well as a finite-dimensional unital associative algebra $\fa$.
  
  If $x\in X$ ranges over a set containing exactly one element in every closed $G$-orbit, the functor
  \begin{equation}\label{eq:main}
    \bigoplus_x {_{\fa^x}}\cM_f\longrightarrow {_\fM}\cM_f
  \end{equation}
  induces a bijection between isomorphism classes of simple modules.
\end{theorem}

The proof uses \emph{descent} for both modules and comodules (\Cref{th:desc1,th:desc2} below, easily recovered from broader Hopf-algebraic results): casting equivariant modules/comodules over a ``larger'' object (such as $B:=A\otimes \fa$ or the regular-function Hopf algebra $\cO(G)$) as non-equivariant modules/comodules over ``smaller'' corresponding objects (e.g. $B^G$ or $\cO(G_x)$ respectively). 

Arbitrary finite-dimensional $\fM$-modules, for that matter, ``specialize well'' in the sense of \Cref{th:direct_sum} below. For a $\Bbbk$-point $x\in X$ with closed orbit $O_x$, denote by $\cM_x$ the full subcategory of $\cM:={_\fM}\cM_f$ consisting of objects $M$ such that $B\otimes_{B^G}M$ is supported on $O_x$. 

\begin{theorem}\label{th:direct_sum}
  If $x\in X$ ranges over a set containing exactly one element in every closed $G$-orbit, the functor
  \begin{equation*}
    \bigoplus_x\cM_x\to \cM
  \end{equation*}
built out of the inclusions $\cM_x\to \cM$ is an equivalence. 
\end{theorem}

\subsection*{Acknowledgments}

I am grateful for input on the literature from E. Neher and A. Savage. 

\section{Preliminaries}\label{se:prel}

\emph{Linear algebraic groups} are as in \cite[Remark 4.11]{mln_alg-gp_2017} (and hence synonymous to \emph{affine} algebraic groups): closed group subschemes $G\le \mathrm{GL}(n)$, neither reduced/smooth nor irreducible in general (by contrast to \cite[\S I.1, 1.1]{Bor91} say, where reduction is assumed). For the little general background and terminology needed here revolving around coalgebras, Hopf algebras, comodules and the like we refer the reader to \cite{abe,dnr,mont,rad,swe}, etc. \emph{$R$-points} on a scheme $Y$ are those belonging to $Y(R)$, when conflating $Y$ with its \emph{functor of points} (\spr{01J5}, \cite[\S 13.1]{Bor91}); this will apply mostly to $R:=\Bbbk$ (in which case it is not uncommon to also refer to these as \emph{$\Bbbk$-rational} points). 

We denote by $\cO(Y)$ the algebra of regular functions on a scheme $Y$. If $G\le \mathrm{GL}(n)$ is a linear algebraic group, then $\cO(G)$ is a Hopf algebra, and $G$-representations are $\cO(G)$-comodules; for this reason, we also write $\rep(G)$ for $\cM^{\cO(G)}$. Recall \cite[\S 1.2, Definition 1.4]{fkm} that $G$ is \emph{linearly reductive} if $\rep(G)$ is semisimple (or: the Hopf algebra $\cO(G)$ is \emph{cosemisimple} \cite[pre Theorem 3.1.5]{dnr}). Superscripts denote invariants:
\begin{equation*}
  \begin{aligned}
    \left(M\xrightarrow{\quad\rho\quad}M\otimes H\right)\in \cM^H
    \quad\left(\text{bialgebra }H\right)
    &\quad:\quad
      M^H
      :=
      \left\{m\in M\ :\ \rho(m)=m\otimes 1\right\}\\
    \left(M\xrightarrow{\quad\rho\quad}M\otimes \cO(G)\right)\in \rep(G)
    &\quad:\quad
      M^G
      :=
      M^{\cO(G)}
      =
      \left\{m\in M\ :\ \rho(m)=m\otimes 1\right\}
  \end{aligned}  
\end{equation*}

For \emph{monoidal categories} \cite[Definition 6.1.1]{brcx_hndbk-1} $(\cC,\otimes,\mathbf{1})$ and \emph{algebras in (or internal to)} $\cC$ \cite[Definition 7.8.1]{egno} we denote by $_{B}\cC$ the category of \emph{$B$-modules in $\cC$} \cite[Definition 7.8.5]{egno}: objects $M\in \cC$ equipped with $\cC$-morphisms $B\otimes M\to M$ unital and associative in the obvious sense. This applies in particular to $\cC:=\cM^H$ for Hopf algebras (or bialgebras $H$), so also to $\cC:=\rep(G)$. 


The theory of algebraic-group \emph{orbits} is developed in \cite[\S 7.c]{mln_alg-gp_2017} (as in \cite[\S 5.3]{dg_gp-alg_1970}, \cite[\S I.1, 1.7]{Bor91}, etc.) in the context of actions on \emph{algebraic} schemes, i.e. \cite[pre \S 1.a]{mln_alg-gp_2017} those of finite type over the ground field. In that setup, regarding a $\Bbbk$-point $x$ as a morphism $\mathrm{Spec}(\Bbbk)\to X$, the orbit $O_x$ is defined as the image of the map 
\begin{equation}\label{eq:orbit}
  \begin{tikzpicture}[auto,baseline=(current  bounding  box.center)]
    \node (1) at (0,0) {$G\cong G\times\mathrm{Spec}(\Bbbk)$};
    \node (2) at (4,0) {$G\times X$};
    \node (3) at (7,0) {$X.$};
    \draw[->] (1) to node {$\id_G\times x$} (2);
    \draw[->] (2) to node {} (3);
  \end{tikzpicture}
\end{equation}  
It is a priori a topological subspace of $X$, but turns out \cite[Proposition 1.65]{mln_alg-gp_2017} to be \emph{locally closed} (i.e. open in its closure); this gives $O_x$ a natural reduced scheme structure. Furthermore, for smooth $G$ and finite-type separated $X$ there is \cite[Proposition 7.17]{mln_alg-gp_2017} an identification
\begin{equation*}
  G/G_x
  \xrightarrow[\quad\cong\quad]{\quad}
  O_x
  \lhook\joinrel\xrightarrow[\quad\text{immersion}\quad]{\quad}
  X
\end{equation*}
with the quotient of $G$ by the \emph{isotropy group} \cite[post Proposition 7.5]{mln_alg-gp_2017} of $x$. This suffices to extend the discussion to possibly-non-algebraic affine $X$, assuming $G$ smooth (equivalently \cite[Proposition 1.26]{mln_alg-gp_2017}, reduced; this is the case we will be interested in):
\begin{itemize}[wide]
\item Write $X=\varprojlim_i X_i$ as a \emph{cofiltered limit} \spr{04AY} of finite-type affine $G$-$\Bbbk$-schemes, dual to the exhaustion $A=\varinjlim_i A_i$ by finitely-generated $G$-subalgebras. This is a limit in the category of $\Bbbk$-schemes, but also that of sets and/or topological spaces: \cite[Tags \href{https://stacks.math.columbia.edu/tag/0CUE}{0CUE} and \href{https://stacks.math.columbia.edu/tag/0CUF}{0CUF}]{stacks-project}. 

\item Writing $x_i$ for the image of $x$ through $X\to X_i$, observe that $G_{x_i}$ stabilizes to $G_x\le G$ for large $i$ by the descending chain condition \cite[Corollary 1.42]{mln_alg-gp_2017} on algebraic subgroups.

\item Limiting over $i$ we obtain a morphism $G/G_x\to X$, which we refer to as the orbit $O_x$.

\item If moreover the orbits $O_{x_i}\subseteq X$ are closed for large $x$ then said morphism is a closed immersion \spr{0CUH}, so the orbit will be a closed subscheme of $X$. This is what is meant below by requiring that $x\in X(\Bbbk)$ have closed orbit. 
\end{itemize}

I will use the following descent results where (as not unusual in category-theoretic literature \cite[Definition 19.3]{ahs}) the tail of the symbol `$\bot$' points towards the left hand of an \emph{adjunction}. 

\begin{theorem}\label{th:desc1}
  \begin{enumerate}[(1),wide]
  \item\label{item:th:desc1:hopf} For a cosemisimple Hopf algebra $H$ be and an algebra $B\in \cM^H$ an algebra
    \begin{equation}\label{eq:desc1.hopf}
      \begin{tikzpicture}[auto,baseline=(current  bounding  box.center)]
        \node (1) at (0,0) {$_{B^H}\cM$};
        \node (bot) at (3,0) {$\bot$};
        \node[text width=4cm] (2) at (8,0) {objects $M\in {_B}\cM^H$ such that $BM^H=M$};
        \draw[->] (1) .. controls (2,1) and (4,1) .. node{$B\otimes_{B^H}\bullet$} (2);
        \draw[<-] (1) .. controls (2,-1) and (4,-1) .. node[below] {$(\bullet)^H$} (2);
      \end{tikzpicture}
    \end{equation}
    
  \item\label{item:th:desc1:gp} In particular if $G$ is a linearly-reductive linear algebraic group and $A\in \rep(G)$ an algebra then
    \begin{equation}\label{eq:desc1}
      \begin{tikzpicture}[auto,baseline=(current  bounding  box.center)]
        \node (1) at (0,0) {$_{B^G}\cM$};
        \node (bot) at (3,0) {$\bot$};
        \node[text width=4cm] (2) at (8,0) {objects $M\in {_B}\rep(G)$ such that $BM^G=M$};
        \draw[->] (1) .. controls (2,1) and (4,1) .. node{$B\otimes_{B^G}\bullet$} (2);
        \draw[<-] (1) .. controls (2,-1) and (4,-1) .. node[below] {$(\bullet)^G$} (2);
      \end{tikzpicture}
    \end{equation}
    is an equivalence of categories.   
  \end{enumerate}
\end{theorem}
\begin{proof}
  \Cref{item:th:desc1:hopf} specializes to \Cref{item:th:desc1:gp} at $H:=\cO(G)$, so we focus on the former.

  That the two functors depicted in \Cref{eq:desc1.hopf} constitute an adjunction between $_{B^H}\cM$ and $_{B}\cM^H$ is well known (\cite[\S 3]{zbMATH04208321} say); we denote it by $F\dashv G$ for brevity.

  \cite[Lemma 3.4]{zbMATH04208321} implies (given the assumed cosemisimplicity) that the \emph{unit} \cite[\S IV.1, post Theorem 1]{mcl_2e} $\id\to GF$ of that adjunction is a natural isomorphism. $F$ is thus \emph{fully faithful} by \cite[Proposition 3.4.1]{brcx_hndbk-1}, and the adjunction restricts to an equivalence between the domain $_{B^H}\cM$ of $F$ and the essential image of $F$. That image is nothing but the category of ${}_{B}\cM^H$-objects $N$ for which the counit $GF\to\ id$ is an isomorphism, i.e. those specified in the statement. 
\end{proof}


\begin{theorem}\label{th:desc2}
  For an affine $\Bbbk$-scheme $X$ acted upon by the linear algebraic $\Bbbk$-group $G$ and $x\in X(\Bbbk)$ with closed orbit $O_x=\mathrm{Spec}(R)$ and isotropy $G_x\le G$ the adjunction
\begin{equation}\label{eq:desc2}
\begin{tikzpicture}[auto,baseline=(current  bounding  box.center)]
  \node (1) at (0,0) {$\rep(G_x)$};
  \node (2) at (8,0) {${_R}\rep(G)$};
  \node (top) at (4,0) {$\top$};
  \draw[->] (1) .. controls (2,1) and (6,1) .. node{induction $G_x$-reps $\to$ $G$-reps} (2);
  \draw[<-] (1) .. controls (2,-1) and (6,-1) .. node[below] {fiber of $R$-module at $x\in\mathrm{Spec}(R)$} (2);
\end{tikzpicture}
\end{equation}
is an equivalence of categories.  
\end{theorem}
\begin{proof}
  Casting $G_x\le G$ as a Hopf quotient $\cO(G)\xrightarrowdbl{} \cO(G_x)$, the claimed adjunction becomes
  \begin{equation*}
    \bigg(-\square_{\cO(G_x)}\cO(G)\bigg)
    \quad
    \vdash
    \quad
    \bigg(\left(R/x\right)\otimes_R-\bigg)
    \quad
    \text{\cite[post Proposition 1]{tak-ff}},
  \end{equation*}
  `$\square$' denoting \emph{cotensoring} \cite[\S 10.1]{bwis}. The assumed orbit affineness is equivalent \cite[Theorem 10]{tak-ff} to the \emph{faithful coflatness} \cite[\S 10.9]{bwis} of $\cO(G)$ as a (left or right) $\cO(G_x)$-comodule, hence the equivalence by \cite[Theorems 1 and 2]{tak-ff}. 
\end{proof}

\section{Main results}\label{se:main}

\subsection{Simple modules}\label{subse.simp}

\Cref{th:main} classifies the simple finite-dimensional $\fM$-modules in terms of the fixed-point subalgebras $\fa^x\le \fa$ for $\Bbbk$-rational points $x\in X$. All conventions set out in \Cref{se:prel} are in place. Recall the evaluation maps $\cM\to\fa^x$; they induce restriction functors ${_{\fa^x}}\cM\to{_\fM}\cM$. 

Set $B:=A\otimes\fa$ so that $B^G=\fM$, and denote by $V$ a finite-dimensional $B^G$-module.


\begin{lemma}\label{le:supp1}
  If $V\in{_\fM}\cM_f$ is simple, then the support of $V'=B\otimes_{B^G}V$ as an $A$-module is a minimal closed $G$-invariant subset of $X$.
\end{lemma}
\begin{proof}
  Since $V$ is finite-dimensional, it is finitely generated over $B^G$. This means that $V'$ is finitely generated over $B=A\otimes\fa$, and hence over $A$ (because $\fa$ is finite-dimensional). Its support must then be closed \cite[Chapter 1, Exercise (2)]{mtsm_ca_2e_1980}, and it is in any case $G$-invariant. 

  An application of \Cref{th:desc1} shows that $V'$ is a simple object in the category $\cC$ on the right hand side of \Cref{eq:desc1}. If $\supp_A(V')$ is not a \define{minimal} closed $G$-invariant subset, then we can find a proper, closed, $G$-invariant subset $Z\subset \supp(V')$ corresponding to some $G$-invariant ideal $I\trianglelefteq A$. The quotient
  \begin{equation*}
    V''
    :=
    (A/I)\otimes_AV'
    =
    V'/IV'
  \end{equation*}
  in $\cC$ is either trivial or full by simplicity, and we have a contradiction:
  \begin{itemize}[wide]
  \item $V''$ cannot vanish unless $V'$ does (and with it also $V$ by \Cref{th:desc1}, in which case there is nothing to prove) by \emph{Nakayama} \spr{07RC} upon localizing at some prime $\fp\in Z$;

  \item while on the other hand $IV'$ cannot vanish: if it did, localization at some $\fp\in \supp(V')\setminus Z$ would annihilate $V'$.
  \end{itemize}
\end{proof}

Minimal closed $G$-sets might of course, in principle, contain no $\Bbbk$-rational points (e.g. $G$ might be trivial with $A$ an infinite field extension of $\Bbbk$). For the supports of \Cref{le:supp1} this is ruled out by the following observation.

\begin{lemma}\label{le:orbit}
  If $V\in{_\fM}\cM_f$ is simple, then the $A$-support of $V'=B\otimes_{B^G}V$ is a closed $G$-orbit in $X$. 
\end{lemma}
\begin{proof}
  Most of what is required already effectively features in the discussion of orbits preceding \Cref{th:desc1}. Write once again
  \begin{equation*}
    X=\varprojlim_i \left(X_i:=\spec\left(A_i\right)\right)
    ,\quad
    A=\bigcup^{\text{filtered union}}_i \left(\text{finitely-generated $G$-invariant }A_i\right),
  \end{equation*}
  ordering $i\le j$ by inclusion $A_i\le A_j$. As $V$ will be simple over $B_i^G$, $B_i:=A_i\otimes \fa$ for sufficiently large $i$, we assume for simplicity that this is the case for all $i$.

  Applying \Cref{le:supp1} at the individual $i$ to $A_i$,
  \begin{equation*}
    O_i
    :=
    \supp_{A_i}\left(V_i':=B_i\otimes_{B_i^G}V\right)
    \subseteq
    \spec(A_i)
  \end{equation*}
  is minimal closed $G$-invariant. $A_i$ being of finite type, $O_i$ must be a closed $G$-orbit (as follows from \cite[Proposition 7.5((b))]{mln_alg-gp_2017} for instance, given the fact that non-empty finite-type $\Bbbk$-schemes have $\Bbbk$-points). I next claim that
  \begin{equation}\label{eq:trnstn}
    \forall\left(i\le j\right)
    \ :\ 
    \pi_{ji}\left(O_{j}\right)
    =
    O_i
    \quad\text{for}\quad
    X_j
    \xrightarrow[\quad\text{transition map}\quad]{\quad\pi_{ji}\quad}
    X_i.
  \end{equation}
  Indeed, it suffices to argue that $\pi_{ji}$ maps maximal ideals $\fm\in O_j$ into $O_i$. This, in turn, follows from the observation that the canonical transition map
  \begin{equation*}
    B_i\otimes_{B_i^G}V
    =
    V'_i
    \xrightarrow{\quad}
    V'_j
    =
    B_j\otimes_{B_j^G}V
  \end{equation*}
  becomes a surjection after respectively quotienting out the kernels of the morphisms $B_i\xrightarrowdbl{}\fa$ induced by $\fm$, and hence if $V'_j$ is not annihilated by that procedure then neither is $V'_i$.
  
  Now, \cite[Proposition 7.12]{mln_alg-gp_2017} respectively identifies $O_i$ with quotients $G/H_i$ for algebraic subgroups $H_i\le G$. Being non-increasing with $i\uparrow$ by \Cref{eq:trnstn}, the $H_i$ stabilize \cite[Corollary 1.42]{mln_alg-gp_2017} to some $H\le G$ and the (co)restrictions $O_j\xrightarrow{\pi_{ji}}O_i$ are isomorphisms for large $i\le j$. This realizes the limit
  \begin{equation*}
    O
    :=
    \varprojlim_i O_i
    \lhook\joinrel\xrightarrow[\quad\text{closed immersion: \spr{0CUH}}\quad]{\quad}
    X=\varprojlim_i X_i
  \end{equation*}
  as a closed $G$-orbit $\cong G/H$ in $X$. 

  To conclude, observe that the minimal $G$-invariant set $Y:=\supp_A(V':=B\otimes_{B^G}V)$ is contained in $O$ (so must coincide with it): for $\fp\in Y$ we have
  \begin{equation*}
    V'_{\fp}
    \cong
    \varinjlim_i
    \left(V'_i\right)_{\fp_i}
    ,\quad
    \left(
      X\ni \fp
      \xmapsto[\quad\text{canonical limit structure map}\quad]{\quad}
      \fp_i\in X_i,
    \right)
  \end{equation*}
  so that $\left(V'_i\right)_{\fp_i}$ must be non-zero for large $i$ if $V'_{\fp}$ is. 
\end{proof}

\begin{remark}
  \Cref{le:orbit} is analogous to \cite[Proposition 5.2]{NehSavSen12}, which proves essentially the same thing for finite $G$ (but not necessarily associative $\fa$). In that case we have at our disposal the result that the fibers of the map $\spec(A)\to\spec(A^G)$ are $G$-orbits; this is more problematic for positive-dimensional $G$. 
\end{remark}

For a point $x\in X$ with closed $G$-orbit $O_x$ let $A_x=\cO(O_x)$ and $B_x=A_x\otimes\fa$. Before moving on to the formal proof of \Cref{th:main}, it might be helpful to note that schematically, the argument moves between the various categories introduced above as indicated in the following diagram:

\begin{equation*}
  \begin{tikzpicture}[auto,baseline=(current  bounding  box.center)]
    \node (1) at (0,0) {${_{B^G}}\cM$};
    \node (2) at (2,2) {${_B}\rep(G)$};
    \node (3) at (6,3) {$\bigoplus_x ({_{B_x}}\rep(G))$};
    \node (4) at (10,2) {$\bigoplus_x ({_\fa}\rep(G_x)$)};
    \node (5) at (12,0) {${_{\fa^x}}\cM$};
    \draw[<->,bend left=15] (1) to node{\Cref{th:desc1}} (2);
    \draw[<->,bend left=10] (2) to node{\Cref{le:orbit}} (3);
    \draw[<->,bend left=10] (3) to node{\Cref{th:desc2}} (4);
    \draw[<->,bend left=15] (4) to node{\Cref{th:desc1}} (5);
    \draw[<->,bend right=05] (1) to (5);
  \end{tikzpicture}
\end{equation*}

\pf{th:main}
\begin{th:main}
For a $\Bbbk$-point $x\in X$ with closed orbit $O_x$ let $\cC_x$ be the full subcategory of ${_B}\rep(G)$ consisting of objects $M$ supported on $O_x$ such that $M^G$ is finite-dimensional and $BM^G=M$.  

According to \Cref{th:desc1} and \Cref{le:orbit} we have a bijection $B\otimes_{B^G}\bullet$ between the (isomorphism classes of) simples in $_{\fM}\cM_f$ and those in the direct sum $\bigoplus_x \cC_x$ (or equivalently in the direct product $\prod_x \cC_x$) for $x$ ranging over any set containing exactly one $\Bbbk$-point from each closed $G$-orbit in $X$. 

Set $H=G_x$, the isotropy group of the $\Bbbk$-point $x\in X$ (whose orbit is assumed to be closed, so that $H$ is again linearly reductive). I now claim that taking the fiber at $x$ produces a bijection between the (isomorphism classes of) simple objects in $\cC_x$ and those in the full subcategory $\cD_x$ of ${_\fa}\rep(H)$ consisting of objects $N$ supported on the orbit $O_x$ with finite-dimensional $N^{H}$ and such that $\fa N^{H}=N$. 

Assuming the claim for now, we can finish the proof of the theorem by applying \Cref{th:desc1} once more to conclude that $(\bullet)^H$ identifies the simples of $\cD_x$ with those of ${_{\fa^x}}\cM_f$. We leave it to the reader to confirm that the identifications we have made are compatible with \Cref{eq:main}. 

It remains to prove the claim. Note first that a simple object in $\cC_x$ is actually a module over the reduced ring $A_x=\cO(O_x)$ (else tensoring with $A_x$ would produce a proper non-zero quotient). Hence, the simples of $\cC_x$ coincide with those in the category of $B_x$-modules $M$ in $\rep(G)$ for which (a) $M^G$ is finite-dimensional and (b) $B_xM^G=M$. The claim now follows from the next lemma applied to $R=A_x=\cO(O_x)$.  
\end{th:main}

\begin{lemma}\label{le:main_claim}
In the setting of \Cref{th:desc2}, let $\fa\in\rep(G)$ be an algebra. Then, the equivalence \Cref{eq:desc2} specializes to an equivalence 
\begin{equation}\label{eq:main_claim}
\begin{tikzpicture}[auto,baseline=(current  bounding  box.center)]
  \node[text width=2.8cm] (1) at (0,0) {$N\in {_\fa}\rep(G_x)$, $\dim(N^{G_x})<\infty$, $\fa N^{G_x}=N$};
  \node[text width=3cm] (2) at (10,0) {$M\in {_{R\otimes\fa}}\rep(G)$, $\dim(M^G)<\infty$, $(R\otimes\fa)M^G=M$};
  \node (top) at (5,0) {$\top$};
  \draw[->] (1) .. controls (2,1) and (8,1) .. node{induction from $G_x$-reps to $G$-reps} (2);
  \draw[<-] (1) .. controls (2,-1) and (8,-1) .. node[below] {fiber of $R$-module at $x\in\mathrm{Spec}(R)$} (2);
\end{tikzpicture}
\end{equation}
between full subcategories of ${_\fa}\rep(G_x)$ and ${_{R\otimes\fa}}\rep(G)$ respectively.  
\end{lemma}
\begin{proof}
  Note first that the equivalence \Cref{eq:desc2} is one of symmetric monoidal categories, where the monoidal structures are the obvious ones (tensoring over $\Bbbk$ on the left and over $R$ on the right in \Cref{eq:desc2}). Since $R\otimes\fa$ is an algebra in ${_R}\rep(G)$ whose image in $\rep(G_x)$ is $\fa$, \Cref{eq:desc2} lifts to an equivalence
  \begin{equation}\label{eq:main_claim_bis}
    \begin{tikzpicture}[auto,baseline=(current  bounding  box.center)]
      \node (1) at (0,0) {${_\fa}\rep(G_x)$};
      \node (2) at (6,0) {${_{R\otimes\fa}}\rep(G)$};
      \node (top) at (3,0) {$\top$};
      \draw[->] (1) .. controls (2,1) and (4,1) .. node{} (2);
      \draw[<-] (1) .. controls (2,-1) and (4,-1) .. node[below] {} (2);
    \end{tikzpicture}
  \end{equation}
  Mapping the object $M$ on the right hand side of \Cref{eq:main_claim_bis} canonically onto its $x$-fiber $N=M\otimes_{R}(R/x)$ identifies $M^G$ and $N^{G_x}$, which implies that the two finiteness conditions in \Cref{eq:main_claim} do indeed coincide.

  Finally, we have to verify that if $M$ on the right hand side of \Cref{eq:main_claim} corresponds to $N$ on the left hand side, then $\fa N^{G_x}=N$ is equivalent to $(R\otimes\fa)M^G=M$. 

  \begin{itemize}[wide]
  \item On the one hand, tensoring $(R\otimes\fa)M^G=M$ with $R/x$ produces $\fa N^{G_x}=N$ (recall that $M^G\cong N^{G_x}$).

  \item Conversely, suppose $\fa N^{G_x}=N$. Then, $(R\otimes\fa)M^G$ is a subobject of $M$ in ${_{R\otimes\fa}}\rep(G)$ whose $x$-fiber is again $N$. But since \Cref{eq:main_claim_bis} is an equivalence, the inclusion $(R\otimes\fa)M^G\le M$ must be an equality.
  \end{itemize}
\end{proof}

\subsection{Arbitrary modules}\label{subse.arb}

According to \Cref{le:orbit} the closed $G$-orbits in $X$ naturally label the simple objects in $\cM:={_\fM}\cM_f$ can be labeled with closed $G$-orbits in $X$. \Cref{th:direct_sum} shows that this labeling can be extended to a direct sum decomposition of the entire category. 

\begin{remark}\label{re.disjoint_images}
  An object $M\in\cM$ is in $\cM_x$ if and only if it is supported on the image $\overline{x}$ of $x$ through $X\to X/G$. 

  Note that the relevant $\Bbbk$-points of $X/G$, i.e. those which are images of closed orbits in $X$, are in bijection with these orbits. To see this, consider two distinct (and hence disjoint) closed $G$-orbits $O_x$ and $O_y$ in $X$. Let $Z=O_x\sqcup O_y$ be the reduced closed subscheme, and $\overline{A}=\cO(Z)$ the corresponding quotient of $A$. By linear reductivity, $A^G\to\overline{A}^G$ is onto. This implies that the lower right hand arrow in      
  \begin{equation*}
    \begin{tikzpicture}[auto,baseline=(current  bounding  box.center)]
      \node (1) at (0,0) {$Z$};
      \node (2) at (2,.5) {$X$};
      \node (3) at (2,-.5) {$Z/G$};
      \node (4) at (4,0) {$X/G$};
      \draw[->,bend left=7] (1) to (2);
      \draw[->,bend left=7] (2) to (4);
      \draw[->,bend right=7] (1) to (3);
      \draw[->,bend right=7] (3) to (4); 
    \end{tikzpicture}
  \end{equation*}
  is one-to-one. Since the lower corner of the diagram is a two-point scheme, we are done. 
\end{remark}

\pf{th:direct_sum}
\begin{th:direct_sum}
  By \Cref{le:orbit} we know that every simple is an object of one of the categories $\cM_x$. Since every object $M$ in $\cM$ is a successive extension of simples, we will be done if we show that there are no non-trivial extensions between simple objects $M,N$ with $B\otimes_{B^G}M$ and $B\otimes_{B^G}N$ supported on different closed orbits $O_x$ and $O_y$ respectively.   

  We have to prove that $\mathrm{Ext}:=\mathrm{Ext}^1_{B^G}(M,N)$ vanishes. Let $\overline{x}$ and $\overline{y}$ be the images of $x$ and $y$ respectively in $X/G=\mathrm{Spec}(A^G)$. They are the supports of $M$ and $N$, and by \Cref{re.disjoint_images} they are distinct. Hence, we can find $f\in A^G$ belonging to the maximal ideal $\overline{y}$ but not to $\overline{x}$.     

  Note that $\mathrm{Ext}$ is acted upon naturally by $A^G$ via its action on either $M$ or $N$. On the one hand the action of $f$ on $N$ is zero, so $f$ annihilates $\mathrm{Ext}$. On the other hand, I claim that $f$ acts as an isomorphism on $M$ and hence on $\mathrm{Ext}$, proving that the latter vanishes. 

  We are left having to check the claim. The annihilator of $f$ in $M$ is an $A^G$-submodule supported on a set strictly smaller than the singleton $\overline{x}$ (because $f\not\in \overline{x}$), which means that the action of $f$ on $M$ is one-to-one; $M$ being finite-dimensional, $f:M\to M$ is also onto. 
\end{th:direct_sum}


\addcontentsline{toc}{section}{References}

\Addresses

\end{document}